\documentclass[11pt]{amsart}
\usepackage{amssymb, latexsym}

\theoremstyle{plain}
\newtheorem{theorem}{Theorem}

\newtheorem{proposition}{Proposition}

\theoremstyle{remark}

\newtheorem*{Remark 1}{Remark 1}
\newtheorem*{Remark 2}{Remark 2}
\newtheorem*{Remark 3}{Remark 3}
\newtheorem*{Remark 4}{Remark 4}

\numberwithin{equation}{section}

\begin{document}

\title[ Random Walk Excited By Its Recent History ]%
 { The Speed of a Random Walk Excited By Its Recent History}

\author{ Ross G. Pinsky}
\address{Department of Mathematics\\
Technion---Israel Institute of Technology\\
Haifa, 32000\\ Israel} \email{ pinsky@math.technion.ac.il}
\urladdr{http://www.math.technion.ac.il/~pinsky/}

\date{}

\begin{abstract}
Let $N$ and $M$ be positive integers satisfying $1\le M\le N$, and let $0<p_0<p_1<1$.
Define a process $\{X_n\}_{n=0}^\infty$ on $\mathbb{Z}$ as follows. At each step, the process jumps
either one step to the right or one step to the left, according to the following mechanism. For the first $N$ steps,
the process behaves like a  random walk that jumps to the right with probability $p_0$ and to the left
with probability $1-p_0$. At subsequent steps the jump
mechanism is defined as follows: if at least $M$ out of the $N$ most recent   jumps were
to the right, then the probability of jumping to the right is $p_1$; however, if fewer than $M$ out of the $N$ most recent  jumps were to the right,
then the probability of jumping to the right is $p_0$.
We calculate the speed of the process. Then we let $N\to\infty$ and  $\frac MN\to r\in[0,1]$, and calculate
the limiting speed. More generally, we consider the above questions for a random walk with a finite number $l$
of threshold levels, $(M_i,p_i)_{i=1}^l$, above the pre-threshold level $p_0$, as well as for one model with $l=N$ such thresholds.
\end{abstract}

\maketitle
\section{Introduction and Statement of Results}\label{S:intro}
Over the past couple of decades, quite a number of papers have been devoted to the study
of edge or vertex reinforced random walks and excited  (also known as ``cookie'') random walks.
These processes have a simple underlying transition mechanism---such as simple symmetric random walk---but this
mechanism is ``reinforced'' or ``excited''  depending on the location of the random walk and its complete  history at that location.
For survey papers which include many references, see \cite{P} and \cite{KZ}.

In this paper, we consider random walks with a different kind of excitement. These
random walks are excited by their recent history, irrespective of their present location.
Let $N$ and $M$ be positive integers satisfying $1\le M\le N$, and let $0<p_0<p_1<1$.
Define a process $\{X_n\}_{n=0}^\infty$ on $\mathbb{Z}$ as follows. At each step, the process jumps
either one step to the right or one step to the left, according to the following mechanism. For the first $N$ steps,
the process behaves like a  standard random walk that jumps to the right with probability $p_0$ and to the left
with probability $1-p_0$. At subsequent steps the jump
mechanism is defined as follows: if at least $M$ out of the $N$ most recent   jumps were
to the right, then the probability of jumping to the right is $p_1$; however, if fewer than $M$ out of the $N$ most recent  jumps were to the right,
then the probability of jumping to the right is $p_0$.
We call this process a \it\ random walk excited by its recent history.\rm\ We call $N$ the \it history window\rm\ for the process, and $\frac MN$ the \it threshold fraction.\rm\

This process and  the more involved multi-stage threshold process defined below are  natural models for the fortunes
of various economic commodities, such as stocks, which respond positively to recent success and negatively to recent failure.
These processes $\{X_n\}_{n=1}^\infty$ are not Markovian, however, the corresponding
$(N+1)$-dimensional processes $\{X_j,X_{j+1},\cdots, X_{j+N}\}_{j=0}^\infty$ are Markovian.
We note that these processes
are examples of random walks \it with internal states\rm; see \cite[page 177]{H}. In particular, if $N=1$, the  process defined in the previous paragraph is known as
a \it persistent random walk\rm; see \cite{H}, \cite{RH}, \cite{HV} and references therein.

We are interested in the deterministic speed of the process, which we denote by $s(N,M;p_0,p_1)$:
$$
s(N,M;p_0,p_1):=\lim_{n\to\infty}\frac{X_n}n\
\text{a.s.}
$$
The following theorem shows that $s(N,M;p_0,p_1)$ exists, and
gives an  explicit expression  for this speed.
\begin{theorem}\label{speed1}
The speed $s(N,M;p_0,p_1)$ exists and is given by
\begin{equation}\label{speed1form}
\begin{aligned}
&s(N,M;p_0,p_1)=\\
&\frac{\frac1{1-p_0}\sum_{j=0}^{M-1}(2\frac jN-1)\binom Nj(\frac{p_0}{1-p_0})^j+\frac1{1-p_1}(\frac{p_0}{1-p_0})^M\sum_{j=M}^N(2\frac jN-1)\binom Nj(\frac{p_1}{(1-p_1)})^{j-M}}
{\frac1{1-p_0}\sum_{j=0}^{M-1}\binom Nj(\frac{p_0}{1-p_0})^j+\frac1{1-p_1}(\frac{p_0}{1-p_0})^M\sum_{j=M}^N\binom Nj
(\frac{p_1}{(1-p_1)})^{j-M}}.
\end{aligned}
\end{equation}
\end{theorem}
\noindent \bf Remark.\rm\ The form of the speed function $s(N,M;p_0,p)$ is quite complicated, even for small values of $N$.
The one case where the form is somewhat simple is the case in which $N=M$, that is, the case in which the process is excited
only at those times at which it has just experienced at least $N$ consecutive jumps to the right.
We have
$$
\sum_{j=0}^{N-1}\binom Nj(\frac{p_0}{1-p_0})^j=\sum_{j=0}^N\binom Nj(\frac{p_0}{1-p_0})^j-(\frac{p_0}{1-p_0})^N=(\frac1{1-p_0})^N-(\frac{p_0}{1-p_0})^N,
$$
and
$$
\sum_{j=0}^Nj\binom Nj(\frac{p_0}{1-p_0})^j=N\frac{p_0}{1-p_0}\sum_{j=1}^N\binom{N-1}{j-1}(\frac{p_0}{1-p_0})^{j-1}=
\frac{Np_0}{1-p_0}(\frac1{1-p_0})^{N-1}.
$$
Making the calculations and doing a little algebra, one finds that
$$
s(N,N;p_0,p_1)=\frac{p_0^N(p_1-p_0)+(1-p_1)(2p_0-1)}{p_0^N(p_1-p_0)+1-p_1}.
$$
\medskip

We now  let   the history window $N$ increase to infinity and we let the threshold fraction $\frac MN$ converge to some
limiting value $r\in[0,1]$. There exists a critical value of $r$ below which the speed
converges to $2p_1-1$ and above which the speed converges to $2p_0-1$.
\begin{theorem}\label{limitspeed1}
Let
$$
r^*(p_0,p_1):=\frac{\log\frac{1-p_0}{1-p_1}}{\log(\frac{p_1}{p_0}\frac{1-p_0}{1-p_1})},\ 0< p_0< p_1<1.
$$
Then $r^*(p_0,p_1)$ is strictly monotone increasing in each of its variables, and $p_0<r^*(p_0,p_1)<p_1$.

One has
$$
\lim_{N\to\infty,\frac MN\to r}s(N,M;p_0,p_1)=\begin{cases} 2p_1-1,& \ \text{if}\ \  0\le r<r^*(p_0,p_1);\\ 2p_0-1,& \ \text{if}\ \ r^*(p_0,p_1)<r\le 1.
\end{cases}
$$
If $\frac MN\to r^*(p_0,p_1)$ and
$\lim_{N\to\infty}\big(\frac{p_0(1-p_1)}{p_1(1-p_0)}\big)^{M-Nr^*(p_0,p_1)}=\alpha\in[0,\infty]$, then one has
$$
\lim_{N\to\infty}s(N,M;p_0,p_1)=
\frac{(2p_0-1)(1-p_1)+(2p_1-1)(1-p_0)\alpha}{1-p_1+(1-p_0)\alpha}.
$$
(If $\alpha=\infty$, the above limit is of course interpreted as $2p_1-1$.)
\end{theorem}
\noindent \bf Remark 1.\rm\ Note that $r^*(p_0,p_1)$ can be characterized as the unique value of $r\in(0,1)$ for which
$(\frac{p_1}{p_0})^r(\frac{1-p_1}{1-p_0})^{1-r}=1$.

\noindent \bf Remark 2.\rm\ For fixed $p_0,p_1$, if $r^*(p_0,p_1)$ is irrational, then  every limiting speed between $2p_0-1$ and $2p_1-1$ is possible, by varying the behavior of
$M$, while if $r^*(p_0,p_1)$ is rational, then a dense countable set of speeds between
 $2p_0-1$ and $2p_1-1$ are possible, by varying the behavior of $M$;  however
the only  speeds that are stable with respect to  small perturbations of $M$ (or $r$) are $2p_0-1$ and  $2p_1-1$.

\noindent \bf Remark 3.\rm\
If $\frac MN$ converges to $r^*(p_0,p_1)$ sufficiently rapidly, then $\alpha=1$ and the limiting speed is
$\frac{1-p_1}{(1-p_1)+(1-p_0)}(2p_0-1)+\frac{1-p_0}{(1-p_1)+(1-p_0)}(2p_1-1)$, which is larger than the average between the speeds $2p_1-1$ and $2p_0-1$, obtained respectively
when $r<r^*(p_0,p_1)$ and when $r>r^*(p_0,p_1)$.


\medskip

We now consider a multi-stage threshold version of the process. Let $l\ge1$ be an integer, denoting the number of threshold
stages. (The case $l=1$ is the case treated above; we include it here so that Theorem \ref{speed2} below will include
 Theorem \ref{speed1} as a particular case and  Theorem \ref{limitspeed2}
 below will include   Theorem
\ref{limitspeed1} (except for the claim regarding $r^*(p_0,p_1)$) as a particular case.)
Let $N\ge l$ denote the history window. Let $\{M_j\}_{j=1}^l$ satisfy
$1\le M_1<\cdots<M_l\le N$ and  let $\{p_j\}_{j=0}^l$
satisfy $0<p_0<p_1<\cdots<p_l<1$. For notational  convenience, define $M_0=0$ and $M_{l+1}=N+1$.
We define the process $\{X_n\}_{n=0}^\infty$ on $\mathbb{Z}$ as follows. At each step, the process jumps
either one step to the right or one step to the left, according to the following mechanism. For the first $N$ steps,
the process behaves like a  random walk that jumps to the right with probability $p_0$ and to the left with probability $1-p_0$. At subsequent steps the jump
mechanism is defined as follows: for  $i=0,1,\cdots, l$,  if  between  $M_i$ and $M_{i+1}-1$ out of the $N$ most recent   jumps were
to the right, then the probability of jumping to the right is $p_i$.
We denote the speed of the process by $s(N,M_1,\cdots, M_l;p_0,\cdots, p_l)$:
$$
s(N,M_1,\cdots, M_l;p_0,\cdots, p_l):=\lim_{n\to\infty}\frac{X_n}n\
\text{a.s.}
$$
Again, we  show that the speed exists and give an explicit expression for it. Throughout the paper, we use the standard convention that a void product of the form $\prod_{k=1}^0$
 is equal to 1.
\begin{theorem}\label{speed2}
The speed $s(N,M_1,\cdots, M_l;p_0,\cdots,p_l)$ exists, and is given by
\begin{equation}\label{speed2form}
\begin{aligned}
&s(N,M_1,\cdots, M_l;p_0,\cdots, p_l)=\\
&\frac{\sum_{i=0}^l\frac1{1-p_i}\big(\prod_{k=1}^i(\frac{p_{k-1}}{1-p_{k-1}})^{M_k-M_{k-1}}\big)
\sum_{j=M_i}^{M_{i+1}-1}(2\frac jN-1)\binom Nj(\frac{p_i}{1-p_i})^{j-M_i}}{\sum_{i=0}^l\frac1{1-p_i}\big(\prod_{k=1}^i(\frac{p_{k-1}}{1-p_{k-1}})^{M_k-M_{k-1}}\big)\sum_{j=M_i}^{M_{i+1}-1}
\binom Nj(\frac{p_i}{1-p_i})^{j-M_i}}.
\end{aligned}
\end{equation}
\end{theorem}
\medskip

As in the case of the  single stage threshold, we now  let   the history window $N$ increase to infinity,
 and we let the threshold fractions converge to limiting values: $\lim_{N\to\infty}\frac {M_i}N=r_i$, $i=1,\cdots, l$.
(The $r_i$'s are not required to be distinct.) For notational convenience, define $r_0=0$ and $r_{l+1}=N$. Recall  the notation
$\text{argmax}_{i\in A}f(i)=\{j\in A: f(j)=\max_{i\in A}f(i)\}$.
\begin{theorem}\label{limitspeed2}
\noindent i.
If
$$
\text{argmax}_{i\in\{0,\cdots, l\}:r_i< p_i< r_{i+1}}\frac1{p_i^{r_i}(1-p_i)^{1-r_i}}\prod_{k=1}^i(\frac{p_{k-1}}{1-p_{k-1}})^{r_k-r_{k-1}}=\{i_0\},
$$
then
\begin{equation}\label{limitspeed2form}
\lim_{N\to\infty,\frac{M_i}N\to r_i,\ i=1,\cdots, l}s(N,M_1,\cdots, M_l;p_0,\cdots, p_l)=2p_{i_0}-1.
\end{equation}
\noindent ii.
If
\begin{equation}\label{argmax}
\text{argmax}_{i\in\{0,\cdots, l\}:r_i< p_i< r_{i+1}}\frac1{p_i^{r_i}(1-p_i)^{1-r_i}}\prod_{k=1}^i(\frac{p_{k-1}}{1-p_{k-1}})^{r_k-r_{k-1}}=
\{i_1,\cdots,  i_d\},
\end{equation}
$d\ge2$, and
\begin{equation}\label{argmaxcondition}
\alpha_{i_j}:=\lim_{N\to\infty}
(\frac{p_{i_j}}{1-p_{i_j}})^{-M_{i_j}+Nr_{i_j}}\prod_{k=1}^{i_j}(\frac{p_{k-1}}{1-p_{k-1}})^{(M_k-Nr_k)-(M_{k-1}-Nr_{k-1})}
\in[0,\infty],
\end{equation}
for   $j=1,\cdots, d$, with at most one $j$ for which $\alpha_{i_j}=\infty$,
then
\begin{equation}\label{multi}
\lim_{N\to\infty,\frac{M_i}N\to r_i,\ i=1,\cdots, l}s(N,M_1,\cdots, M_l;p_0,\cdots, p_l)=
\frac{\sum_{j=1}^d\frac{\alpha_{i_j}}{1-p_{i_j}}
(2p_{i_j}-1)}
{\sum_{j=1}^d\frac{\alpha_{i_j}}{1-p_{i_j}}}.
\end{equation}
\end{theorem}
\noindent \bf Remark 1.\rm\
Similar to the situation described in Remark 2 after Theorem \ref{limitspeed1},
for fixed $\{p_i\}_{i=0}^l$, depending on their particular values, either
  every limiting speed between $2p_0-1$ and $2p_l-1$ is possible or a countably dense set of speeds
  between $2p_0-1$ and $2p_l-1$ are possible, by varying the behavior
of $\{M_i\}_{i=1}^l$, however
the only   speeds that are stable with respect to small perturbations of $\{M_i\}_{i=1}^l$ (or $\{r_i\}_{i=1}^l$)
are   $\{2p_i-1\}_{i=0}^l$. Furthermore, for fixed $\{p_i\}_{i=0}^l$ and $\{r_i\}_{i=1}^l$, if the speed is stable then it must be
from among the speeds
$\{2p_i-1\}_{i: r_i< p_i< r_{i+1}}$. We note that necessarily there is at least one $i$ for which $r_i<p_i<r_{i+1}$.

\noindent \bf Remark 2.\rm\ The requirement in part (ii)  that at most one of the $\alpha_{i_j}=\infty$  was made
in order to avoid complications in the statement of the theorem. The method of proof of part (ii) would also show,
for example, that if more than one of the $\alpha_{i_j}$'s are equal to infinity, and there is a particular $j_0$ such that
the order of  $(\frac{p_{i_j}}{1-p_{i_j}})^{-M_{i_j}+Nr_{i_j}}\prod_{k=1}^{i_j}(\frac{p_{k-1}}{1-p_{k-1}})^{(M_k-Nr_k)-(M_{k-1}-Nr_{k-1})}$
 as $N\to\infty$ is larger when $j=j_0$  than it is for any other $j$, then the limiting speed
is $2p_{i_{j_0}}-1$.
One can  refine this further when two or more of the terms above are on the same order.

\noindent \bf Remark 3.\rm\ In part (ii) of the theorem, the set $\{i_1,\cdots, i_d\}$ need not consist of consecutive
integers. For example, consider $l=2$ and denote the expression corresponding to $i$  appearing on the left hand side of \eqref{argmax} by $J_i$.
Then one has $J_0=\frac1{1-p_0}$, $J_1=\frac1{p_1^{r_1}(1-p_1)^{1-r_1}}(\frac{p_0}{1-p_0})^{r_1}=\frac1{1-p_1}(\frac{p_0}{p_1})^{r_1}(\frac{1-p_1}{1-p_0})^{r_1}$,
and similarly, after some algebra, $J_2=(\frac{p_1}{p_2})^{r_2}(\frac{1-p_2}{1-p_1})^{r_2}\frac{1-p_1}{1-p_2}J_1$.
Choose $0<p_0<p_1<1$ and choose $r_1$ to satisfy $p_0<r^*(p_0,p_1)<r_1<p_1$, where $r^*(p_0,p_1)$ is as in Theorem \ref{limitspeed1}.  Then $J_0>J_1$.
Since $\lim_{p_2\to p_1}(\frac{p_1}{p_2})^{r_2}(\frac{1-p_2}{1-p_1})^{r_2}\frac{1-p_1}{1-p_2}=1$,   uniformly over $r_2\in[0,1]$,
and since  $\lim_{p_2\to1}(\frac{p_1}{p_2})^{r_2}(\frac{1-p_2}{1-p_1})^{r_2}\frac{1-p_1}{1-p_2}=\infty$, for all $r_2\in[0,1)$, we can choose $r_2>p_1$ sufficiently close to $p_1$ to guarantee
that there exists $p_2$ such that $p_1<r_2<p_2<r_3=1$ and such that $J_2=J_0$.
Then the left hand side of \eqref{argmax} will be  equal to $\{i_0,i_2\}$.

\noindent \bf Remark 4.\rm\ Fix $p_0$ as the  pre-threshold stage.
It is possible to have a situation in which each individual threshold stage, $(M_i,p_i)$, is such
that if it were the only threshold stage, then the limiting speed as the history window increases to infinity would be $2p_0-1$, however
when all the stages are present the limiting speed is larger than this. For example, let $l=2$ and  consider for simplicity the case $p_0=\frac12$; so $2p_0-1=0$.
The condition $r<r^*(\frac12,p)$ in Theorem \ref{limitspeed1} can be written as
$\frac1{p^r(1-p)^{1-r}}>2$ (see Remark 1 after Theorem \ref{limitspeed1}).
Thus,  if
$\frac1{p_i^{r_i}(1-p_i)^{1-r_i}}<2$, for  $i= 1,2$, then from Theorem \ref{limitspeed1}, if $(M_i,p_i)$ were the only threshold stage (and $p_0=\frac12$),  the limiting speed would be zero.
One can arrange things so that the above conditions hold but such that
 $\max_{i\in\{0,1,2\}:r_i<p_i<r_{i+1}}\frac1{p_i^{r_i}(1-p_i)^{1-r_i}}\prod_{k=1}^i(\frac{p_{k-1}}{1-p_{k-1}})^{r_k-r_{k-1}}>2$, in which case
 the limiting speed is greater than zero, since the term on the left hand side above,
over which the maximum is being taken, is equal to 2 when $i=0$.
Indeed, for $i=2$, the expression above over which the maximum is being taken is
$\frac1{p_2^{r_2}(1-p_2)^{1-r_2}}( \frac{p_1}{1-p_1})^{r_2-r_1}$.
From the strict monotonicity of $r^*(\frac12,p)$ in $p$,
we can select $p_1,p_2$ and $r_1$ such that $\frac12<p_1<p_2$ and  $r^*(\frac12,p_1)<r_1<r^*(\frac12,p_2)$.
Then if we choose $r_2>r^*(\frac12,p_2)$ such that $r_2-r^*(\frac12,p_2)$
is sufficiently small, we will have
$\frac1{p_2^{r_2}(1-p_2)^{1-r_2}}( \frac{p_1}{1-p_1})^{r_2-r_1}>2$ and $r_2<p_2<r_3=1$.

\noindent \bf Remark 5.\rm\ Consider the random walk with $l$ threshold stages $(M_i,p_i)$ with
 $\frac {M_i}N\to r_i$, $1\le i\le l$. Now consider a random walk with  $(l+1)$ threshold stages, with the $l$ stages above and with an
additional stage $(M,p)$ satisfying $M_{i_1}<M<M_{i_1+1}$, $p_{i_1}<p<p_{i_1+1}$ and
$\frac MN\to r$, with $r_{i_1}<r<r_{i_1+1}$, for some $i_1\in\{0,1,\cdots, l\}$.
Theorem \ref{limitspeed2} allows one to make
a comparison between  the limiting speeds of the two processes as the history window increases to infinity.
Assume that the condition in part (i) of the theorem is in effect so that
the  limiting speed for the $l$-stage process is $2p_{i_0}-1$. If  $i_0>i_1$, then the limiting speed of the $(l+1)$-stage
process will also be $2p_{i_0}-1$; that is, the additional stage cannot affect the limiting speed. On the other hand,
if  $i_0\le i_1$, then the limiting speed of the $(l+1)$-stage
process might be greater than that of the $l$-stage process and might be equal to that of the $l$-stage process,  depending
on the particular values of the relevant parameters.

\medskip

We now present a weak convergence result which shows that in the case that the limiting speed as $N\to\infty$ is stable,
  the random walk excited by its recent history, and shifted sufficiently forward in time, converges weakly
  to a simple random walk.
\begin{theorem}\label{weakconvergence}
Let $P_N$ denote probabilities for the
 random walk  excited by its recent history with history window $N$ and multistage thresholds $\{M_i\}_{i=1}^l$.
Assume that
$\frac {M_i}N\to r_i$, $i=1,\cdots, l$, and that
\begin{equation}\label{argmaxassump}
\text{argmax}_{i\in[0,\cdots, l]:r_i< p_i< r_{i+1}}\frac1{p_i^{r_i}(1-p_i)^{1-r_i}}\prod_{k=1}^i(\frac{p_{k-1}}{1-p_{k-1}})^{r_k-r_{k-1}}=\{i_0\},
\end{equation}
in which case the limiting speed as $N\to\infty$ is $2p_{i_0}-1$.
If the deterministic times $\{T_N\}_{N=1}^\infty$ converge to infinity sufficiently fast,
then the process
 $\{X_{N+T_N+n}-X_{N+T_N}\}_{n=1}^\infty$  converges weakly
to the simple random walk with probability $p_{i_0}$ of
jumping to the right and $1-p_{i_0}$ of jumping to the left. That is,
\begin{equation}\label{weakly}
\begin{aligned}
&\lim_{N\to\infty}P_N(X_{N+T_N+1}-X_{N+T_N}=e_1,\cdots,X_{N+T_N+m}-X_{N+T_N+m-1}=e_m)=\\
&p_{i_0}^{|\{i\in[m]:\thinspace e_i=1\}|}(1-p_{i_0})^{|\{i\in[m]:\thinspace e_i=-1\}|},
\ \text{for all}\  m\ge1\ \text{and}\
 (e_1,\cdots, e_m)\in\{-1,1\}^m.
\end{aligned}
\end{equation}
\end{theorem}
\bf \noindent Remark.\rm\ The necessary rate of convergence to infinity of $\{T_N\}_{N=1}^\infty$ is related
to the mixing time of the auxiliary Markov process defined in section \ref{auxiliary}.

The model treated in the theorems above
 involves a fixed number $l$ of threshold stages, independent of $N$.
We now consider a class of models in which the number of threshold stages is  $N$.
Let $G:[0,1]\to(0,1)$ be continuous and nondecreasing. We consider the random walk
excited by its recent history with history window $N$ and with $N$ threshold levels $\{M_i\}_{i=1}^N$,
defined by $M_i=i$. The
corresponding probabilities $\{p_i\}_{i=1}^N$ are given
by $p_i=G(\frac iN)$, $i=1,\cdots, N$.
It follows from Theorem \ref{speed2} that for fixed $N$, the limiting speed, which we will denote by $s(N;G)$, exists
and is given by
\begin{equation*}
s(N;G)=\frac{\sum_{i=0}^N\frac1{1-G(\frac iN)}\big(\prod_{k=1}^i\frac{G(\frac{k-1}N)}{1-G(\frac{k-1}N)}\big)\binom Ni(2\frac iN-1)}
{\sum_{i=0}^N\frac1{1-G(\frac iN)}\big(\prod_{k=1}^i\frac{G(\frac{k-1}N)}{1-G(\frac{k-1}N)}\big)\binom Ni}.
\end{equation*}
(In  Theorem \ref{speed2}, use
$l=N$, $M_j=j$, $1\le j\le N$, and $p_j=G(\frac jN)$, $0\le j\le N$.)
Note that from the assumptions on $G$, it must have at least one fixed point; that is, a point $p\in(0,1)$ for which $G(p)=p$.
We will prove the following result.
\begin{theorem}\label{Gthm}
Let
$$
F(p):=\int_0^p\log\frac{G(x)}{1-G(x)}dx-p\log p-(1-p)\log(1-p), \ 0<p<1.
$$
Then $F$ attains a maximum, and  any point at which $F$ attains its maximum is a fixed point for $G$.
In particular then, a sufficient condition for $F$ to attain its maximum uniquely is that   $G$ have a unique fixed point.
If $F$ attains its maximum uniquely at $p^*$, then
\begin{equation}\label{speed*}
\lim_{N\to\infty}s(N;G)=2p^*-1.
\end{equation}
\end{theorem}
An interesting model, which constitutes a randomization of the basic single threshold model, can be fit into
 the $N$-threshold  model treated in Theorem \ref{Gthm}.
Let $N\ge1$ be a positive integer and let $0<\rho_0<\rho_1<1$.
Define a process $\{X_n\}_{n=0}^\infty$ on $\mathbb{Z}$ as follows. At each step, the process jumps
either one step to the right or one step to the left, according to the following mechanism. For the first $N$ steps,
the process behaves like a  random walk that jumps to the right with probability $\rho_0$ and to the left
with probability $1-\rho_0$. At subsequent steps, the jump mechanism is defined as follows: uniformly at random select one step
from among the most recent $N$ steps. If that step was a jump to the right, then the probability of jumping to the right
is $\rho_1$, while if that step was a jump to the left, then the probability of jumping to the right is $\rho_0$.
This model is in fact equivalent to the $N$-threshold model with $G$ being the linear function $G(p)=p\rho_1+(1-p)\rho_0$.
Indeed, note that for any $0\le j\le N$, if in fact $j$ of the $N$ most recent jumps were to the right, then
when selecting one step from among the most recent $N$ steps uniformly at random, the probability that that step
will be a step to the right is $\frac jN$.
Thus, if $j$ of the $N$ most recent jumps were to the right, the probability of jumping to the right will be
$\frac jN\rho_1+(1-\frac jN)\rho_0$.
The  above function $G$ has a unique fixed point  $p^*=\frac{\rho_0}{1-\rho_1+\rho_0}$.
Thus, for this model, the limiting speed as $N\to\infty$ is
$2\frac{\rho_0}{1-\rho_1+\rho_0}-1$.

In the above case, when $G$ is linear, the result in Theorem \ref{Gthm} can be understood heuristically as follows.
Let
$G(p)=p\rho_1+(1-p)\rho_0$. Fix any $p\in(0,1)$ and consider a simple random walk with probability
$p$ of jumping to the right. If after many (at least $N$) steps, we were to impose on the process for one step
the $N$-threshold model with this $G$, then
the probability of that step being to the right would be $G(p)=p\rho_1+(1-p)\rho_0$. If we were then to   let
the process continue as a simple random walk with probability $G(p)$ of jumping right, and then again after many steps
we were to impose the above mechanism  for one step, then the probability of the next step being to the right would be $G(G(p))=G^{(2)}(p)$.
Comparing this process to the actual process, a little thought suggests that
as $N\to\infty$, the limiting speed of the actual process should be $2p^*-1$, where $p^*$ is the fixed point of $G$.

\medskip

As already noted,  Theorem \ref{speed2}  includes
 Theorem \ref{speed1} as a particular case and  Theorem \ref{limitspeed2}
  includes   Theorem
\ref{limitspeed1} (except for the claim regarding $r^*(p_0,p_1)$) as a particular case.
In section \ref{auxiliary} we define and compute the invariant probability measure of an  auxiliary Markov chain that encodes
the $N$ most recent jumps of the original process.  Using this result, the proof of Theorem \ref{speed2} is almost
immediate; it appears in section \ref{proofspeed2}. In section \ref{prooflimitspeed2} we prove part (i)
of Theorem \ref{limitspeed2} and  in section \ref{limitcaseprob} we prove part (ii) of  Theorem \ref{limitspeed2}.
Theorem \ref{weakconvergence} is proved in section \ref{weakconv} and Theorem \ref{Gthm} is proved
in section \ref{lequalsN}.
Finally, we prove here the statement in Theorem \ref{limitspeed1} about the behavior of $r^*(p_0,p_1)$; namely that it is strictly
monotone increasing in each of its variables, when $p_0<p_1$,   and that it satisfies $p_0<r^*(p_0,p_1)<p_1$, when $p_0<p_1$.
We have
$$
\begin{aligned}
&\frac{\partial r^*}{\partial p_0}(p_0,p_1)=\frac1{1-p_0}\frac1{(\log\frac{p_1}{p_0}\frac{1-p_0}{1-p_1})^2}\big(-\log\frac{p_1}{p_0}+
\frac{1-p_0}{p_0}\log\frac{1-p_0}{1-p_1}\big):=\\
&\frac1{1-p_0}\frac1{(\log\frac{p_1}{p_0}\frac{1-p_0}{1-p_1})^2}R(p_0,p_1).
\end{aligned}
$$
It is easy to see that
$\lim_{p_0\to0}R(p_0,p_1)=\infty$,
$\lim_{p_0\to p_1^-}R(p_0,p_1)=0$ and  $\frac{\partial R}{\partial p_0}(p_0,p_1)<0$.
Thus $r^*(p_0,p_1)$ is strictly monotone increasing in $p_0$. A similar calculation shows that
 $r^*(p_0,p_1)$ is strictly monotone increasing in $p_1$.
Using l'H\^opital's rule, it follows easily that
$$
\lim_{p_0\to p_1^-}r^*(p_0,p_1)=\lim_{p_1\to p_0^+}r^*(p_0,p_1)=p_0=p_1.
$$
This in conjunction with the monotonicity shows that
$p_0<r^*(p_0,p_1)<p_1$, when $p_0<p_1$.

\section{An auxiliary Markov chain and its invariant measure}\label{auxiliary}
Let $\{X_n\}_{n=0}^\infty$ denote the random walk excited by its recent history, with $l$ threshold stages, $l\ge1$.
Define
the $N$-dimensional process $\{Z_n\}_{n=0}^\infty$ by
$$
Z_n=(X_{n+1}-X_n, X_{n+2}-X_{n+1},\cdots, X_{n+N}-X_{n+N-1}).
$$
The process $\{Z_n\}_{n=0}^\infty$ encodes the most recent $N$ jumps of the original process $\{X_n\}_{n=0}^\infty$.
It is clear from the definition of the original process that  $\{Z_n\}_{n=0}^\infty$ is a Markov process; its state space is $H_N:=\{-1,1\}^N$. Denote probabilities for this Markov process
starting from $v\in H_N$ by $P_v$, and denote the corresponding expectation by $E_v$.
Clearly, $\{Z_n\}_{n=0}^\infty$ is irreducible.
For $v=(v_1,\cdots, v_N)\in H_N$, let $\#^+(v)=\sum_{i=1}^N 1_{\{v_i=1\}}$ denote the number of 1's from among its $N$ entries.
It turns out that the invariant probability measure $\mu_Z$ on $H_N$ for the process
$\{Z_n\}_{n=0}^\infty$ is constant on the level sets of $\#^+$.
It is this fact
that allows for the explicit calculation of the speed in Theorem \ref{speed2}.
Recall that we have defined for notational convenience $M_0=0$ and $M_{l+1}=N+1$.
\begin{proposition}\label{invariant}
For $i=0,1,\cdots, l$, one has
\begin{equation}\label{explicitinvar}
\begin{aligned}
&\mu_Z(v)=C\frac{1}{(1-p_i)}(\frac{p_i}{1-p_i})^{\#^+(v)-M_i}\prod_{k=1}^i(\frac{p_{k-1}}{1-p_{k-1}})^{M_k-M_{k-1}},\\
&\text{if}\ M_i\le \#^+(v)<M_{i+1},
\end{aligned}
\end{equation}
where $C$ is the appropriate normalizing constant.
\end{proposition}
\begin{proof}
Let $A=\{A_{w,v}\}_{w,v\in H_N}$ denote the transition probability matrix for the Markov chain $\{Z_n\}_{n=0}^\infty$; that is,
$A_{w,v}=P(Z_{n+1}=v|Z_n=w)$.
The invariant probability measure $\mu_Z=\mu_Z(v)$ is the unique function satisfying $\mu_ZA=\mu_Z$ and $\sum_{v\in H_N}\mu_Z(v)=1$.
The equality $\mu_ZA=\mu_Z$ is a set of equations indexed by $v\in H_N$ via the term $\mu_Z(v)$ on the right hand side of the equality.
We will attempt to find a solution to the above set of equations with $\mu_Z$ constant on the level sets of $\#^+$.

Let $\alpha_k=\mu_Z(v)$, for $\#^+(v)=k$, $k=0,1,\cdots, N$. We substitute this into the above equations.
Consider first the equation corresponding to a $v\in H_N$  satisfying
$\#^+(v)=k$ with $1\le k\le M_1-1$ and $v_N=1$. There are exactly two states
$w$ that lead to $v$; namely $w_1=(1,v_1,\cdots, v_{N-1})$ and
$w_2=(-1,v_1,\cdots, v_{N-1})$. We have $\#^+(w_1)=k$ and $\#^+(w_2)=k-1$.
Because we have assumed that $k\le M_1-1$, we have $\#^+(w_i)\le M_1-1$, $i=1,2$.
Consequently,  conditioned on $Z_n=w_i$, it follows that at time $n+N$, no more than $M_1-1$ of the most recent $N$ jumps  of the process $\{X_k\}_{k=0}^\infty$ were to the right.
This means that  conditioned on $Z_n=w_i$, the probability that $X_{N+n+1}=X_{N+n}+1$ is equal to $p_0$;  equivalently, $A_{w_i,v}=p_0$.
Thus, from the equation indexed by $v$ in the equality   $\mu_ZA=\mu_Z$,  and from  the definition of $\{\alpha_i\}$, we obtain
\begin{equation}\label{-1}
p_0\alpha_k+p_0\alpha_{k-1}=\alpha_k,\ 1\le k\le M_1-1.
\end{equation}
From this we conclude that
\begin{equation}\label{0}
\alpha_k=(\frac{p_0}{1-p_0})^k\alpha_0,\ k=0,\cdots, M_1-1.
\end{equation}

Now consider the equation corresponding to a $v\in H_N$ satisfying
$\#^+(v)=k\le M_1-2$ and $v_N=-1$. There are exactly two states
$w$ that lead to $v$; namely $w_1=(1,v_1,\cdots, v_{N-1})$ and
$w_2=(-1,v_1,\cdots, v_{N-1})$. We have $\#^+(w_1)=k+1$ and $\#^+(w_2)=k$.
Because we have assumed that $k\le M_1-2$, we have $\#^+(w_i)\le M_1-1$, $i=1,2$.
Consequently,  as in the previous case, conditioned on $Z_n=w_i$, it follows that at time $n+N$, no more than $M_1-1$ of the most recent $N$ jumps  of the process $\{X_k\}_{k=0}^\infty$ were to the right.
This means that  conditioned on $Z_n=w_i$, the probability that $X_{N+n+1}=X_{N+n}-1$ is equal to $1-p_0$;  equivalently, $A_{w_i,v}=1-p_0$.
Thus, from the equation indexed by $v$ in the equality   $\mu_ZA=\mu_Z$,  and from  the definition of $\{\alpha_i\}$, we obtain
\begin{equation*}
(1-p_0)\alpha_{k+1}+(1-p_0)\alpha_k=\alpha_k, \ 0\le k\le M_1-2,
\end{equation*}
which is consistent with \eqref{-1}.

Now consider the equation corresponding to a $v\in H_N$ satisfying $\#^+(v)=M_1-1$ and $v_N=-1$.
There are exactly two states
$w$ that lead to $v$; namely $w_1=(1,v_1,\cdots, v_{N-1})$ and
$w_2=(-1,v_1,\cdots, v_{N-1})$. We have $\#^+(w_1)=M_1$ and $\#^+(w_2)=M_1-1$.
Consequently, conditioned on $Z_n=w_1$, it follows that at time $n+N$, exactly  $M_1$ of the most recent $N$ jumps  of the process $\{X_k\}_{k=0}^\infty$ were to the right.
This means that  conditioned on $Z_n=w_1$, the probability that $X_{N+n+1}=X_{N+n}-1$ is equal to $1-p_1$;  equivalently, $A_{w_1,v}=1-p_1$.
However, conditioned on  $Z_n=w_2$, it follows that at time $n+N$, exactly  $M_1-1$ of the most recent $N$
jumps  of the process $\{X_k\}_{k=0}^\infty$ were to the right.
This means that  conditioned on $Z_n=w_2$, the probability that $X_{N+n+1}=X_{N+n}-1$ is equal to $1-p_0$;  equivalently, $A_{w_2,v}=1-p_0$.
Thus, from the equation indexed by $v$ in the equality   $\mu_ZA=\mu_Z$,  and from  the definition of $\{\alpha_i\}$, we obtain
\begin{equation}\label{m1}
(1-p_1)\alpha_{M_1}+(1-p_0)\alpha_{M_1-1}=\alpha_{M_1-1}.
\end{equation}
From this and \eqref{0} we conclude that
\begin{equation}\label{01}
\alpha_{M_1}=\frac{p_0}{1-p_1}\alpha_{M_1-1}=\frac{1-p_0}{1-p_1}(\frac{p_0}{1-p_0})^{M_1}\alpha_0.
\end{equation}

Now consider the equation corresponding to  a $v\in H_N$ satisfying $\#^+(v)=k$ with $M_1+1\le k\le M_2-1$ and $v_N=1$.
As before, the two states that lead to $v$ are
$w_1=(1,v_1,\cdots, v_{N-1})$ and
$w_2=(-1,v_1,\cdots, v_{N-1})$. We have $M_1\le \#^+(w_i)\le M_2-1$, for $i=1,2$.
Consequently, conditioned on $Z_n=w_i$, it follows that at time $n+N$, between   $M_1$ and $M_2-1$ of the most recent $N$ jumps  of the process $\{X_k\}_{k=0}^\infty$ were to the right.
This means that  conditioned on $Z_n=w_i$, the probability that $X_{N+n+1}=X_{N+n}+1$ is equal to $p_1$;  equivalently, $A_{w_i,v}=p_1$.
Thus, from the equation indexed by $v$ in the equality   $\mu_ZA=\mu_Z$,  and from  the definition of $\{\alpha_i\}$, we obtain
\begin{equation}\label{intermediate}
p_1\alpha_k+p_1\alpha_{k-1}=\alpha_k,\ M_1+1\le k\le M_2-1.
\end{equation}
From this we conclude that
$\alpha_{M_1+j}=(\frac{p_1}{1-p_1})^j\alpha_{M_1}$, for $j=1,\cdots, M_2-1-M_1$. In conjunction with
\eqref{01}, this gives
\begin{equation}\label{1}
\alpha_{M_1+j}=\frac{1-p_0}{1-p_1}(\frac{p_0}{1-p_0})^{M_1}(\frac{p_1}{1-p_1})^j\alpha_0, \ j=1,\cdots, M_2-1-M_1.
\end{equation}

Now consider the equation corresponding to  a $v\in H_N$ satisfying $\#^+(v)=k$ with $M_1\le k\le M_2-2$ and $v_N=-1$.
The two states that lead to $v$ are
$w_1=(1,v_1,\cdots, v_{N-1})$ and
$w_2=(-1,v_1,\cdots, v_{N-1})$. We have $M_1\le \#^+(w_i)\le M_2-1$, for $i=1,2$.
So by the same reasoning as  in the previous case, we obtain
\begin{equation*}
(1-p_1)\alpha_k+(1-p_1)\alpha_{k+1}=\alpha_k, \ M_1\le k\le M_2-2,
\end{equation*}
which is consistent with \eqref{intermediate}.

Now consider the  equation corresponding to  a $v\in H_N$ satisfying $\#^+(v)=M_1$ and $v_N=1$.
The two states that lead to $v$ are
$w_1=(1,v_1,\cdots, v_{N-1})$ and
$w_2=(-1,v_1,\cdots, v_{N-1})$. We have $\#^+(w_1)=M_1$ and $\#^+(w_2)=M_1-1$.
Thus, the same type of reasoning as above gives
\begin{equation*}
p_1\alpha_{M_1}+p_0\alpha_{M_1-1}=\alpha_{M_1}.
\end{equation*}
This is consistent with \eqref{m1}.

Now consider the  equation corresponding to  a $v\in H_N$ satisfying $\#^+(v)=M_2-1$ and $v_N=-1$.
The two states that lead to $v$ are
$w_1=(1,v_1,\cdots, v_{N-1})$ and
$w_2=(-1,v_1,\cdots, v_{N-1})$. We have $\#^+(w_1)=M_2$ and $\#^+(w_2)=M_2-1$.
Thus, the same type of reasoning as above gives
\begin{equation*}
(1-p_2)\alpha_{M_2}+(1-p_1)\alpha_{M_2-1}=\alpha_{M_2-1}.
\end{equation*}
This in conjunction with \eqref{1} gives
\begin{equation}\label{12}
\begin{aligned}
&\alpha_{M_2}=\frac{p_1}{1-p_2}\alpha_{M_2-1}=\frac{p_1}{1-p_2}\frac{1-p_0}{1-p_1}(\frac{p_0}{1-p_0})^{M_1}(\frac{p_1}{1-p_1})^{M_2-M_1-1}\alpha_0=\\
&\frac{1-p_0}{1-p_2}(\frac{p_0}{1-p_0})^{M_1}(\frac{p_1}{1-p_1})^{M_2-M_1}\alpha_0.
\end{aligned}
\end{equation}
Note that
\eqref{0},\eqref{01},\eqref{1} and \eqref{12} are consistent with the claim of the proposition.
(The factor $1-p_0$ is absorbed in the constant $C$.)
Continuing in this vein completes the proof of the proposition.
\end{proof}

\section{Proof of Theorem \ref{speed2}}\label{proofspeed2}
For $v=(v_1,\cdots, v_N)\in H_N=\{-1,1\}^N$, let
$f(v)=\sum_{i=1}^Nv_i=2\#^+(v)-N$.
By the ergodic theorem,
\begin{equation}\label{ergodic}
\lim_{n\to\infty}\frac1n\sum_{j=0}^{n-1}f(Z_{jN})=\sum_{v\in H_N}f(v)\mu_Z(v)\ \text{a.s.}
\end{equation}
Since $Z_k=(X_{k+1}-X_k,X_{k+2}-X_{k+1},\cdots, X_{k+N}-X_{k+N-1})$,
we have $f(Z_{jN})=X_{(j+1)N}-X_{jN}$, and thus
$\sum_{j=0}^{n-1}f(Z_{jN})=X_{nN}-X_0$.
Therefore, we conclude from \eqref{ergodic} that
\begin{equation}\label{ergodicagain}
\lim_{n\to\infty}\frac{X_{nN}}{nN}=\frac1N\sum_{v\in H_N}f(v)\mu_Z(v)=\sum_{v\in H_N}(2\frac{\#^+(v)}N-1)\mu_Z(v).
\end{equation}
There are $\binom Nj$ different states $v\in H_N$ satisfying $\#^+(v)=j$.
Using this fact along with  \eqref{ergodicagain} and the formula for $\mu_Z$ in \eqref{explicitinvar},
we conclude that
$\lim_{n\to\infty}\frac{X_{nN}}{nN}$ converges almost surely to the expression on the right hand side of
\eqref{speed2form}. We leave to the reader the standard, easy argument to go from the almost sure convergence
of $\lim_{n\to\infty}\frac{X_{nN}}{nN}$ to the almost sure convergence
of $\lim_{m\to\infty}\frac{X_{m}}{m}$.
  \hfill $\square$

\section{Proof of part (i)  of Theorem \ref{limitspeed2}}\label{prooflimitspeed2}
To prove part (i)  of the theorem, we need to determine which of the $l+1$ terms
$\left\{(\prod_{k=1}^i(\frac{p_{k-1}}{1-p_{k-1}})^{M_k-M_{k-1}})\sum_{j=M_i}^{M_{i+1}-1}
\binom Nj(\frac{p_i}{1-p_i})^{j-M_i}\right\}_{i=0}^l$ in the denominator of \eqref{speed2form}  dominates
as $N\to\infty$. (We have ignored the factor $\frac1{1-p_i}$  which does not depend on $N$.)
Let $\{Y_n\}_{n=1}^\infty$ be a sequence of IID Ber$(\frac12)$ random variables ($P(Y_n=1)=P(Y_n=0)=\frac12$),
 and let $S_n=\sum_{k=1}^n Y_n$.
For $i=0,1,\cdots, l$, we write
\begin{equation}\label{asexpectation}
\begin{aligned}
&\sum_{j=M_i}^{M_{i+1}-1}
\binom Nj(\frac{p_i}{1-p_i})^{j-M_i}=2^NE\left((\frac{p_i}{1-p_i})^{S_N-M_i}; M_i\le S_N\le M_{i+1}-1\right)=\\
&2^NE\left(e^{(S_N-M_i)\log\frac{p_i}{1-p_i}}; M_i\le S_N\le M_{i+1}-1\right).
\end{aligned}
\end{equation}
As is well-known \cite[page 35]{DZ}, the large deviations
rate function for $\{S_n\}_{n=1}^\infty$ is given by
$I(s)=\log2s^s(1-s)^{1-s}$, $s\in[0,1]$.
It follows then \cite[page 19]{DZ} that for $0\le a<b\le 1$,
\begin{equation}\label{cramer}
\lim_{N\to\infty}\frac1N\log P(\frac1NS_N\in(a,b))=\lim_{N\to\infty}\frac1N\log P(\frac1NS_N\in[a,b])=-\min_{s\in[a,b]}I(s).
\end{equation}
Recall that $\frac{M_j}N\to r_j$ when $N\to\infty$. Then using \eqref{cramer}
and applying the standard Laplace asymptotic method to \eqref{asexpectation}, it follows  that
\begin{equation}\label{Laplace}
\lim_{N\to\infty}\frac1N\log 2^{-N}\sum_{j=M_i}^{M_{i+1}-1}
\binom Nj(\frac{p_i}{1-p_i})^{j-M_i}=\sup_{s\in[r_i,r_{i+1}]}\left(-I(s)+(s-r_i)\log\frac{p_i}{1-p_i}\right).
\end{equation}
The function $-I(s)+(s-r_i)\log\frac{p_i}{1-p_i}$ is concave for $s\in[0,1]$ and attains its maximum at $s=p_i$.
Thus, it follows that
$\sup_{s\in[r_i,r_{i+1}]}\left(-I(s)+(s-r_i)\log\frac{p_i}{1-p_i}\right)$
is attained at $s=p_i$, if $r_i\le p_i\le r_{i+1}$; at $s=r_i$, if $r_i>p_i$, and at $s=r_{i+1}$, if $r_{i+1}<p_i$.
Substituting these values of $s$ in $-I(s)+(s-r_i)\log\frac{p_i}{1-p_i}$ and doing a little algebra reveals that
\begin{equation}\label{options}
\begin{aligned}
&\sup_{s\in[r_i,r_{i+1}]}\left(-I(s)+(s-r_i)\log\frac{p_i}{1-p_i}\right)=\\
&\begin{cases}&-\log2p_i^{r_i}(1-p_i)^{1-r_i}, \ \text{if}\ r_i\le p_i\le r_{i+1};\\
&-\log2r_i^{r_i}(1-r_i)^{1-r_i},\ \text{if}\ r_i>p_i;\\
&-\log2r_{i+1}^{r_{i+1}}(1-r_{i+1})^{1-r_{i+1}}(\frac{1-p_i}{p_i})^{r_{i+1}-r_i},\ \text{if}\ r_{i+1}<p_i.\end{cases}
\end{aligned}
\end{equation}
Thus, letting
\begin{equation*}
\hat J_i:= \lim_{N\to\infty}\frac1N\log 2^{-N}\Big(\prod_{k=1}^i(\frac{p_{k-1}}{1-p_{k-1}})^{M_k-M_{k-1}}\Big)\sum_{j=M_i}^{M_{i+1}-1}
\binom Nj(\frac{p_i}{1-p_i})^{j-M_i},
\end{equation*}
and letting
$$
J_i:=e^{\hat J_i},\ i=0,\cdots, l,
$$
we have from \eqref{Laplace}, \eqref{options} and the definition of $I(s)$ that
\begin{equation}\label{J}
J_i=\begin{cases}\frac1{2p_i^{r_i}(1-p_i)^{1-r_i}}\prod_{k=1}^i(\frac{p_{k-1}}{1-p_{k-1}})^{r_k-r_{k-1}},\ \text{if}\
r_i\le p_i\le r_{i+1};\\ \frac1{2r_i^{r_i}(1-r_i)^{1-r_i}}\prod_{k=1}^i(\frac{p_{k-1}}{1-p_{k-1}})^{r_k-r_{k-1}},\
\text{if}\ r_i>p_i;\\
\frac1{2r_{i+1}^{r_{i+1}}(1-r_{i+1})^{1-r_{i+1}}}(\frac{p_i}{1-p_i})^{r_{i+1}-r_i}\prod_{k=1}^i(\frac{p_{k-1}}{1-p_{k-1}})^{r_k-r_{k-1}},\ \text{if}\ r_{i+1}<p_i.
\end{cases}
\end{equation}

If $\max_{0\le i\le l}J_i$ occurs uniquely at $i=i_0$, then it follows from \eqref{speed2form} and the definition of the $\{J_i\}$ that the limiting
speed is given by
\begin{equation*}
\begin{aligned}
&\lim_{N\to\infty}s(N,M_1,\cdots, M_l;p_0,\cdots, p_l)=\\
&\lim_{N\to\infty}
\frac{\sum_{j=M_{i_0}}^{M_{i_0+1}-1}(2\frac jN-1)\binom Nj(\frac{p_{i_0}}{1-p_{i_0}})^{j-M_{i_0}}}
{\sum_{j=M_{i_0}}^{M_{i_0+1}-1}
\binom Nj(\frac{p_{i_0}}{1-p_{i_0}})^{j-M_{i_0}}}.
\end{aligned}
\end{equation*}
From the Laplace asymptotic method noted above, it follows that the right hand side above is equal
to $2s_{i_0}^*-1$, where
$$
s_i^*=\text{argmax}_{s\in[r_i,r_{i+1}]}\left(-I(s)+(s-r_i)\log\frac{p_i}{1-pi}\right).
$$
In particular, if $r_{i_0}\le p_{i_0}\le r_{i_0+1}$, then from the penultimate sentence above  \eqref{options} we conclude that
this limiting  speed is $2p_{i_0}-1$. For an alternative more explicit derivation of this, see the proof of
\eqref{Laplacemethod} in
 section \ref{limitcaseprob}.

In light of the above  paragraph, \eqref{limitspeed2form} will follow under
the assumption of part (i)
if we show that if  $\max_{0\le i\le l}J_i$  occurs at some $i_0$, then
$r_{i_0}< p_{i_0}< r_{i_0+1}$.
We won't assume
 that this maximum occurs uniquely at $i_0$; thus,  the proof will allow one also to restrict the maximum in part (ii)
to those $i$ for which $r_i<p_i<r_{i+1}$. (We note that for the argument used in the previous paragraph, it would be
enough to show that $r_i\le p_i\le r_{i+1}$. However, for the method used in the proof of part (ii) in section \ref{limitcaseprob}, we need
$r_i< p_i< r_{i+1}$.)

We first show that it is not possible to have $r_{i_0+1}\le p_{i_0}$.
Assume that $r_{i_0+1}\le p_{i_0}$. Then
 $i_{0}\le l-1$, since $r_{l+1}=1$.
We will show that $J_{i_0+1}>J_{i_0}$, contradicting the definition of $i_0$.
From \eqref{J} we have
\begin{equation}\label{Ji0}
J_{i_0}=\frac1{2r_{i_0+1}^{r_{i_0+1}}(1-r_{i_0+1})^{1-r_{i_0+1}}}
\prod_{k=1}^{i_0+1}(\frac{p_{k-1}}{1-p_{k-1}})^{r_k-r_{k-1}}.
\end{equation}
From \eqref{J}, the value of $J_{i_0+1}$ depends on whether $r_{i_0+2}<p_{i_0+1}$ or
$r_{i_0+1}< p_{i_0+1}\le r_{i_0+2}$.

Consider first the case that $r_{i_0+2}<p_{i_0+1}$.
Then from \eqref{J},
\begin{equation}\label{Ji01}
J_{i_0+1}=\frac1{2r_{i_0+2}^{r_{i_0+2}}(1-r_{i_0+2})^{1-r_{i_0+2}}}
(\frac{p_{i_0+1}}{1-p_{i_0+1}})^{r_{i_0+2}-r_{i_0+1}}\prod_{k=1}^{i_0+1}(\frac{p_{k-1}}{1-p_{k-1}})^{r_k-r_{k-1}}.
\end{equation}
Also,
\begin{equation*}
(\frac{p_{i_0+1}}{1-p_{i_0+1}})^{r_{i_0+2}-r_{i_0+1}}>(\frac{r_{i_0+2}}{1-r_{i_0+2}})^{r_{i_0+2}-r_{i_0+1}};
\end{equation*}
thus
\begin{equation}\label{aux}
\begin{aligned}
&\frac1{r_{i_0+2}^{r_{i_0+2}}(1-r_{i_0+2})^{1-r_{i_0+2}}}
(\frac{p_{i_0+1}}{1-p_{i_0+1}})^{r_{i_0+2}-r_{i_0+1}}>\frac1{r_{i_0+2}^{r_{i_0+1}}(1-r_{i_0+2})^{1-r_{i_0+1}}}>\\
&\frac1{r_{i_0+1}^{r_{i_0+1}}(1-r_{i_0+1})^{1-r_{i_0+1}}},
\end{aligned}
\end{equation}
where the last inequality follows from the fact that $x^{r_{i_0+1}}(1-x)^{1-r_{i_0+1}}$ is decreasing in $x$
for $r_{i_0+1}\le x\le 1$.
From \eqref{Ji0}-\eqref{aux} it follows that $J_{i_0+1}>J_{i_0}$.

Now consider the case that $r_{i_0+1}< p_{i_0+1}\le r_{i_0+2}$.
From \eqref{J} we have
\begin{equation}\label{Ji01again}
J_{i_0+1}=\frac1{2p_{i_0+1}^{r_{i_0+1}}(1-p_{i_0+1})^{1-r_{i_0+1}}}\prod_{k=1}^{i_0+1}(\frac{p_{k-1}}{1-p_{k-1}})^{r_k-r_{k-1}}.
\end{equation}
From \eqref{Ji0}, \eqref{Ji01again} and the above noted monotonicity of $x^{r_{i_0+1}}(1-x)^{1-r_{i_0+1}}$,
it follows that $J_{i_0+1}>J_{i_0}$
This completes the proof that $p_{i_0}<r_{i_0+1}$.

We now show that it is not possible to have $r_{i_0}\ge p_{i_0}$. We assume that $i_0\ge1$, since this is trivial for $i_0=0$.
Assume that $r_{i_0}\ge p_{i_0}$.
We will show that $J_{i_0-1}>J_{i_0}$, contradicting the definition of $i_0$.
From \eqref{J}, we have
\begin{equation}\label{Ji0again}
J_{i_0}=\frac1{2r_{i_0}^{r_{i_0}}(1-r_{i_0})^{1-r_{i_0}}}\prod_{k=1}^{i_0}(\frac{p_{k-1}}{1-p_{k-1}})^{r_k-r_{k-1}}.
\end{equation}
From \eqref{J}, the value of $J_{i_0-1}$ depends on whether
$r_{i_0-1}\le p_{i_0-1}< r_{i_0}$ or
 $r_{i_0-1}>p_{i_0-1}$.

Consider first the case  that $r_{i_0-1}\le p_{i_0-1}< r_{i_0}$.
Then from \eqref{J}
\begin{equation}\label{Ji0-}
J_{i_0-1}=\frac1{2p_{i_0-1}^{r_{i_0-1}}(1-p_{i_0-1})^{1-r_{i_0-1}}}\prod_{k=1}^{i_0-1}(\frac{p_{k-1}}{1-p_{k-1}})^{r_k-r_{k-1}}.
\end{equation}
We will show that
\begin{equation}\label{compare0}
\frac1{p_{i_0-1}^{r_{i_0-1}}(1-p_{i_0-1})^{1-r_{i_0-1}}}>
\frac1{r_{i_0}^{r_{i_0}}(1-r_{i_0})^{1-r_{i_0}}}(\frac{p_{i_0-1}}{1-p_{i_0-1}})^{r_{i_0}-r_{i_0-1}}.
\end{equation}
From \eqref{Ji0again}-\eqref{compare0} it follows that
$J_{i_0-1}>J_{i_0}$, if   $r_{i_0-1}\le p_{i_0-1}<r_{i_0}$.
Now  \eqref{compare0} is equivalent to
\begin{equation}\label{compare}
\frac1{p_{i_0-1}^{r_{i_0}}(1-p_{i_0-1})^{1-r_{i_0}}}>\frac1{r_{i_0}^{r_{i_0}}(1-r_{i_0})^{1-r_{i_0}}}.
\end{equation}
But \eqref{compare} follows from the fact that the function $x^{r_{i_0}}(1-x)^{1-r_{i_0}}$ is increasing in $x$ for $0\le x\le r_{i_0}$,
 and by the assumption that  $p_{i_0-1}< r_{i_0}$.
This completes the case $r_{i_0-1}\le p_{i_0-1}<r_{i_0}$.

Now consider the case $r_{i_0-1}>p_{i_0-1}$.
By \eqref{J} we have
\begin{equation}\label{Ji0-again}
J_{i_0-1}=\frac1{2r_{i_0-1}^{r_{i_0-1}}(1-r_{i_0-1})^{1-r_{i_0-1}}}
\prod_{k=1}^{i_0-1}(\frac{p_{k-1}}{1-p_{k-1}})^{r_k-r_{k-1}}.
\end{equation}
Then $J_{i_0-1}> J_{i_0}$ will follow from \eqref{Ji0-again} and \eqref{Ji0again} if we show that
\begin{equation}\label{lastcase}
\frac1{r_{i_0-1}^{r_{i_0-1}}(1-r_{i_0-1})^{1-r_{i_0-1}}}>\frac1{r_{i_0}^{r_{i_0}}(1-r_{i_0})^{1-r_{i_0}}}(\frac{p_{i_0-1}}{1-p_{i_0-1}})^{r_{i_0}-r_{i_0-1}}.
\end{equation}
Since $r_{i_0-1}>p_{i_0-1}$, \eqref{lastcase} will follow if we show that
\begin{equation*}
\frac1{r_{i_0-1}^{r_{i_0-1}}(1-r_{i_0-1})^{1-r_{i_0-1}}}>\frac1{r_{i_0}^{r_{i_0}}(1-r_{i_0})^{1-r_{i_0}}}(\frac{r_{i_0-1}}{1-r_{i_0-1}})^{r_{i_0}-r_{i_0-1}},
\end{equation*}
or equivalently, that
\begin{equation}\label{lastcaseverify}
\frac1{r_{i_0-1}^{r_{i_0}}(1-r_{i_0-1})^{1-r_{i_0}}}>\frac1{r_{i_0}^{r_{i_0}}(1-r_{i_0})^{1-r_{i_0}}}.
\end{equation}
Since $r_{i_0-1}<r_{i_0}$,  \eqref{lastcaseverify} holds for the same reason that \eqref{compare} holds.
\hfill $\square$
\medskip

\section{Proof of part (ii)  of Theorem \ref{limitspeed2}}\label{limitcaseprob}
We assume that \eqref{argmax} holds. From
 the proof of part (i) of Theorem \ref{limitspeed2}, this means that
 of the  $l+1$ summands in the denominator
of \eqref{speed2form}, labeled from 0 to $l$, the ones with the labels $\{i_1,\cdots,  i_d\}$ dominate over the others when
$N\to\infty$. Note  that $r_{i_t}<p_{i_t}<r_{i_t+1}$, for $t=1,\cdots, d$. (The proof of part (i) established
that this is a necessary condition for domination.)
It then follows from \eqref{speed2form} that
\begin{equation}\label{Lapmult}
\begin{aligned}
&\lim_{N\to\infty}s(N,M_1,\cdots, M_l;p_0,\cdots, p_l)=\\
&\lim_{N\to\infty}\frac{\sum_{t=1}^d\frac1{1-p_{i_t}}\big(\prod_{k=1}^{i_t}(\frac{p_{k-1}}{1-p_{k-1}})^{M_k-M_{k-1}}\big)
\sum_{j=M_{i_t}}^{M_{i_t+1}-1}(2\frac jN-1)\binom Nj(\frac{p_{i_t}}{1-p_{i_t}})^{j-M_{i_t}}}{\sum_{t=1}^d\frac1{1-p_{i_t}}
\big(\prod_{k=1}^{i_t}(\frac{p_{k-1}}{1-p_{k-1}})^{M_k-M_{k-1}}\big)\sum_{j=M_{i_t}}^{M_{i_t+1}-1}
\binom Nj(\frac{p_{i_t}}{1-p_{i_t}})^{j-M_{i_t}}}.
\end{aligned}
\end{equation}

In order to evaluate \eqref{Lapmult}, we need to analyze more closely the behavior of these $d$ summands from the denominator
of \eqref{Lapmult}.
For this analysis, we consider the following setup.
Let $\{Y_n\}_{n=1}^\infty$ be a sequence of IID Ber$(\frac12)$ random variables ($P(Y_n=1)=P(Y_n=0)=\frac12$),
 and let $S_n=\sum_{k=1}^n Y_k$.
For $p\in\{p_{i_1},p_{i_2},\cdots, p_{i_d}\}$, let $\{Y^{*,p}_n\}_{n=1}^\infty$ be a sequence of IID Ber$(p)$ random variables
($P(Y^{*,p}=1)=1-P(Y^{*,p}=0)=p$) defined on the same space
as $\{Y_n\}_{n=1}^\infty$ (so that we can use the same $P$ for probabilities and the same $E$ for  expectations) and let
$S^{*,p}_n=\sum_{k=1}^n Y^{*,p}_k$.
As is well-known,
\begin{equation}\label{RN}
P(S_N^{*,p}\in \cdot)=\frac{E\left((\frac p{1-p})^{S_N}; S_N\in \cdot\right)}{E(\frac p{1-p})^{S_N}}.
\end{equation}
To see this, note  that the Radon-Nikodym derivative $F_1(y)$  of the Ber$(p)$ distribution with respect to the the Ber$(\frac12)$ distribution
is given by $F_1(0)=2(1-p)$ and $F_1(1)=2p$, and consequently, the Radon-Nikodym derivative $F_N(y)$ of the $N$-fold product
of the Ber$(p)$ distribution with respect to the $N$-fold product of the Ber$(\frac12)$ distribution is given by $F_N(y_1,\cdots, y_N)=2^Np^{s_N}(1-p)^{N-s_N}$,
where $s_N=\sum_{i=1}^Ny_i$, and $y_i\in\{0,1\}$, for $i=1,\cdots, N$.
Thus,
$$
P(S_N^{*,p}\in \cdot)=E(2^Np^{S_N}(1-p)^{N-S_N}; S_N\in\cdot)=(2(1-p))^NE(\frac p{1-p})^{S_N}; S_N\in \cdot),
$$ from which
\eqref{RN} follows since
\begin{equation}\label{mgf}
E(\frac p{1-p})^{S_N}=(E(\frac p{1-p})^{Y_1})^N=(\frac12\frac p{1-p}+\frac12)^N=(2(1-p))^{-N}.
\end{equation}

With this setup, we can analyze the summands in the denominator of \eqref{Lapmult}.
We have
\begin{equation}\label{sumexpress}
\begin{aligned}
&\sum_{j=M_{i_t}}^{M_{i_t+1}-1}
\binom Nj(\frac{p_{i_t}}{1-p_{i_t}})^{j-M_{i_t}}=2^{N}E\big((\frac{p_{i_t}}{1-p_{i_t}})^{S_N-M_{i_t}};M_{i_t}\le S_N\le M_{i_t+1}-1\big)=\\
&2^N(\frac{p_{i_t}}{1-p_{i_t}})^{-M_{i_t}}E(\frac{p_{i_t}}{1-p_{i_t}})^{S_N}P(M_{i_t}\le S_N^{*,p_{i_t}}\le M_{i_t+1}-1)=\\
&(\frac{p_{i_t}}{1-p_{i_t}})^{-M_{i_t}}(\frac1{1-p_{i_t}})^N\big(1+o(1)\big),\ \text{as}\ N\to\infty,
\end{aligned}
\end{equation}
where the last equality follows \eqref{mgf} and  from the fact that by the law of large numbers, $\frac{S_N^{*,p_{i_t}}}N$ converges almost surely to $p_{i_t}$
as $N\to\infty$, while
\begin{equation}\label{Mr}
\lim_{N\to\infty}\frac{M_{i_t}}N=r_{i_t}<p_{i_t}<r_{i_t+1}=\lim_{N\to\infty}\frac{ M_{i_t+1}}N.
\end{equation}
Let
\begin{equation}\label{J's}
L_{i_t}:=\big(\prod_{k=1}^{i_t}(\frac{p_{k-1}}{1-p_{k-1}})^{M_k-M_{k-1}}\big)\sum_{j=M_{i_t}}^{M_{i_t+1}-1}
\binom Nj(\frac{p_{i_t}}{1-p_{i_t}})^{j-M_{i_t}}.
\end{equation}
Then from \eqref{sumexpress} we have
\begin{equation}\label{longcalc}
\begin{aligned}
&L_{i_t}=
(\frac{p_{i_t}}{1-p_{i_t}})^{-M_{i_t}}(\frac1{1-p_{i_t}})^N\big(1+o(1)\big)\prod_{k=1}^{i_t}
(\frac{p_{k-1}}{1-p_{k-1}})^{M_k-M_{k-1}}=\\
&\Big((\frac{p_{i_t}}{1-p_{i_t}})^{-r_{i_t}}(\frac1{1-p_{i_t}}) \prod_{k=1}^{i_t}(\frac{p_{k-1}}{1-p_{k-1}})^{r_k-r_{k-1}}\Big)^N\times\\
&(\frac{p_{i_t}}{1-p_{i_t}})^{-M_{i_t}+Nr_{i_t}}(1+o(1))\prod_{k=1}^{i_t}(\frac{p_{k-1}}{1-p_{k-1}})^{(M_k-Nr_k)-(M_{k-1}-Nr_{k-1})}.
\end{aligned}
\end{equation}
Since we are assuming that \eqref{argmax} holds, the terms
$$
(\frac{p_{i_t}}{1-p_{i_t}})^{-r_{i_t}}(\frac1{1-p_{i_t}}) \prod_{k=1}^{i_t}(\frac{p_{k-1}}{1-p_{k-1}})^{r_k-r_{k-1}}=
\frac1{p_{i_t}^{r_{i_t}}(1-p_{i_t})^{1-r_{i_t}}}\prod_{k=1}^{i_t}(\frac{p_{k-1}}{1-p_{k-1}})^{r_k-r_{k-1}}
$$
are identical, for $t=1,\cdots, d$.
Since we are assuming that \eqref{argmaxcondition} holds, it then follows from
  \eqref{longcalc}
that
\begin{equation}\label{ratios}
\lim_{N\to\infty}\frac{L_{i_t}}{L_{i_s}}=\frac{\alpha_{i_t}}{\alpha_{i_s}},\ s,t\in\{1,\cdots, d\}.
\end{equation}

Similar to \eqref{sumexpress}, we have
\begin{equation}\label{withj}
\begin{aligned}
&\sum_{j=M_i}^{M_{i+1}-1}(2\frac jN-1)\binom Nj(\frac{p_i}{1-p_i})^{j-M_i}=\\
&2^N(\frac{p_{i_t}}{1-p_{i_t}})^{-M_{i_t}}E(\frac{p_{i_t}}{1-p_{i_t}})^{S_N}E\big((2\frac{S_N^{*,p_{i_t}}}N-1);M_{i_t}\le S_N^{*,p_{i_t}}\le M_{i_t+1}-1)\big).
\end{aligned}
\end{equation}
Using \eqref{sumexpress} and \eqref{withj} along with  \eqref{Mr} and the fact that  $\frac{S_N^{*,p_{i_t}}}N$ converges almost surely to $p_{i_t}$,
we have
\begin{equation}\label{Laplacemethod}
\begin{aligned}
&\lim_{N\to\infty}\frac{\sum_{j=M_i}^{M_{i+1}-1}(2\frac jN-1)\binom Nj(\frac{p_i}{1-p_i})^{j-M_i}}{\sum_{j=M_{i_t}}^{M_{i_t+1}-1}
\binom Nj(\frac{p_{i_t}}{1-p_{i_t}})^{j-M_{i_t}}}=\\
&\lim_{N\to\infty}\frac{E\big((2\frac{S_N^{*,p_{i_t}}}N-1);M_{i_t}\le S_N^{*,p_{i_t}}\le M_{i_t+1}-1\big)}{P(M_{i_t}\le S_N^{*,p_{i_t}}\le M_{i_t+1}-1)}=2p_{i_t}-1.
\end{aligned}
\end{equation}
Now \eqref{multi} follows from \eqref{Lapmult}, \eqref{J's}, \eqref{ratios} and \eqref{Laplacemethod}.
\hfill $\square$

\section{Proof of Theorem \ref{weakconvergence}}\label{weakconv}
Fix $m\ge1$. We need to prove \eqref{weakly}.
Recall from section \ref{auxiliary}  the auxiliary process $\{Z_n\}_{n=0}^\infty$ defined by
$Z_n=(X_{n+1}-X_n, X_{n+2}-X_{n+1},\cdots, X_{n+N}-X_{n+N-1})$, and recall its
 invariant measure $\mu_Z$, defined on $\{-1,1\}^N$, and  given explicitly in \eqref{explicitinvar}.
In this section we write $\{X_n^N\}_{n=0}^\infty$, $\{Z_n^N\}_{n=0}^\infty$ and   $\mu_{Z^N}$ to emphasize the dependence on $N$.

For $\delta>0$, let
$$
A_{N,i_0,\delta}=\{v\in \{-1,1\}^N:\frac{\#^+(v)}N\in(p_{i_0}-\delta,p_{i_0}+\delta)\}.
$$
From   \eqref{explicitinvar} and  \eqref{argmaxassump},
 along with the Laplace asymptotic method
used  in section \ref{prooflimitspeed2},
it follows that
\begin{equation}\label{to0}
\lim_{N\to\infty}\mu_{Z^N}(A_{N,i_0,\delta})=1.
\end{equation}
We have  $r_{i_0}<p_{i_0}<r_{i_0+1}$ since we are assuming  \eqref{argmaxassump}.
Fix once and for all $\epsilon>0$ sufficiently small so that $r_{i_0}<p_{i_0}-2\epsilon$ and $p_{i_0}+2\epsilon<r_{i_0+1}$.
By the ergodic theorem, we can choose $\{T_N\}_{N=1}^\infty$ so that
$$
\lim_{N\to\infty}|P(Z^N_{T_N}\in A_{N,i_0,\epsilon})-\mu_{Z^N}(A_{N,i_0,\epsilon})|=0.
$$
Thus, from \eqref{to0} we have
\begin{equation}\label{ergthm}
\lim_{N\to\infty}P(Z^N_{T_N}\in A_{N,i_0,\epsilon})=1.
\end{equation}

Since $m$ is fixed, it follows from the form of $Z^N$ that
\begin{equation}\label{sets}
\{Z^N_{T_N}\in A_{N,i_0,\epsilon}\}\subset
\cap_{j=0}^{m-1}\{Z^N_{T_N+j}\in A_{N,i_0,2\epsilon}\},\ \text{for sufficiently large}\ N.
\end{equation}
Since $\lim_{N\to\infty}\frac{M_i}N=r_i$, it follows from \eqref{sets} and the choice of $\epsilon$ that if $N$ is sufficiently large, then
the condition $\{Z^N_{T_N}\in A_{N,i_0,\epsilon}\}$ guarantees that
\begin{equation}\label{MM}
M_{i_0}\le \#^+(Z^N_{T_N+j})<M_{i_0+1},\  \text{for}\ j=0,\cdots, m-1.
\end{equation}

If $M_{i_0}\le \#^+(Z^N_{T_N+j})<M_{i_0+1}$ occurs, then from the connection between $Z^N$ and $X^N$ and the definition of $X^N$,
it follows that the random variable $X^N_{T_N+N+j+1}-X^N_{T_N+N+j}$ will be distributed according to the Ber$(p_{i_0})$ distribution,
for $j=0,\cdots, m-1$.
Note that the  random variables $\{X^N_{N+T_N+j+1}-X^N_{N+T_N+j}\}_{j=0}^{m-1}$ are independent of the random vector
 $Z^N_{T_N}$. From these facts and \eqref{sets}, we conclude that for sufficiently large $N$,
\begin{equation}\label{ifthen}
\begin{aligned}
&\text{if}\
Z^N_{T_N}\in A_{N,i_0,\epsilon},\ \text{then}\\
&P_N(X_{N+T_N+1}-X_{N+T_N}=e_1,\cdots,X_{N+T_N+m}-X_{N+T_N+m-1}=e_m)=\\
&p_{i_0}^{|\{i\in[m]:\thinspace e_i=1\}|}(1-p_{i_0})^{|\{i\in[m]:\thinspace e_i=-1\}|},\
\text{for}\ (e_1,\cdots, e_m)\in\{-1,1\}^m.
\end{aligned}
\end{equation}
Now \eqref{weakly} follows from \eqref{ifthen} and \eqref{ergthm}.
\hfill $\square$

\section{Proof of Theorem \ref{Gthm}}\label{lequalsN}
As noted before the statement of the theorem, $s(N;G)$ is given by Theorem \ref{speed2}, using
$l=N$, $M_i=i$, $1\le i\le N$, and $p_i=G(\frac iN)$, $0\le i\le N$.
We have
\begin{equation}\label{speedN}
s(N;G)=\frac{\sum_{i=0}^N\frac1{1-G(\frac iN)}\big(\prod_{k=1}^i\frac{G(\frac{k-1}N)}{1-G(\frac{k-1}N)}\big)\binom Ni(\frac {2i}N-1)}
{\sum_{i=0}^N\frac1{1-G(\frac iN)}\big(\prod_{k=1}^i\frac{G(\frac{k-1}N)}{1-G(\frac{k-1}N)}\big)\binom Ni}.
\end{equation}
In order to apply the Laplace asymptotic method, we need to analyze the behavior of the  weights
$\big(\prod_{k=1}^i\frac{G(\frac{k-1}N)}{1-G(\frac{k-1}N)}\big)\binom Ni$ as $N\to\infty$.

Let
\begin{equation}\label{defF}
F_{N,i}=\prod_{k=1}^i\frac{G(\frac{k-1}N)}{1-G(\frac{k-1}N)}=\prod_{j=0}^{i-1}\frac{G(\frac jN)}{1-G(\frac jN)},\ \text{for}\ 0\le i\le N.
\end{equation}
For $\gamma\in(0,1)$, we have
\begin{equation}\label{Fasymp}
\begin{aligned}
&\log F_{N,[\gamma N]}=N\big(\frac1N\sum_{j=0}^{[\gamma N]-1}\log\frac{ G(\frac jN)}{1-G(\frac jN) }\big)=\\
&N\Big(\int_0^\gamma\log\frac{G(x)}{1-G(x)}dx+o(1)\Big),\ \text{as}\ N\to\infty.
\end{aligned}
\end{equation}
Thus,
\begin{equation}\label{Fbehav}
F_{N,[\gamma N]}=
e^{N\cdot o(1)}e^{N\int_0^\gamma\log\frac{G(x)}{1-G(x)}dx},\
\text{as}\ N\to\infty.
\end{equation}

Stirlings' formula gives
\begin{equation}\label{Stirling}
\binom N{[\gamma N]}=\frac1{\sqrt{\gamma(1-\gamma)N}}\big(\frac1{\gamma^\gamma(1-\gamma)^{1-\gamma}}\big)^Ne^{O(1)},\ \text{as}\ N\to\infty,
\end{equation}
with $e^{O(1)}$ uniform over $\gamma$ in compact subintervals of $(0,1)$.
Also, if $\lim_{N\to\infty}\gamma_N=0$, then for some universal constants $c,C>0$,
$$
\begin{aligned}
&\binom N{[\gamma_N N]}\le \frac{N^{\gamma_NN}}{[\gamma_NN]!}\le\frac{cN^{\gamma_NN}}{(\gamma_NN-1)^{\gamma_NN-1}e^{-\gamma_NN}}\le
\frac{CN^{\gamma_NN+1}e^{\gamma_NN}}{(\gamma_NN)^{\gamma_NN}}=\\
&CNe^{\gamma_NN}\gamma_N^{-\gamma_NN}.
\end{aligned}
$$
The logarithm of the right hand side above is $o(N)$ while the logarithm of the right hand
side of \eqref{Stirling} is on the order $N$; thus, for any $\gamma\in(0,1)$, $\binom N{[\gamma N]}$
is on a larger order as $N\to\infty$ than is $\binom N{[\gamma_N N]}$ with $\lim_{N\to\infty}\gamma_N=0$
or (in light of the symmetry) with  $\lim_{N\to\infty}\gamma_N=1$.

From \eqref{defF}, \eqref{Fbehav} and \eqref{Stirling}, we conclude that
\begin{equation}\label{weights}
\begin{aligned}
&\big(\prod_{k=1}^{[\gamma N]}\frac{G(\frac{k-1}N)}{1-G(\frac{k-1}N)}\big)\binom N{[\gamma N]}=
\frac{e^{O(1)}e^{N\cdot o(1)}}{\sqrt{\gamma(1-\gamma)N}}\big(\frac{e^{\int_0^\gamma\log\frac{G(x)}{1-G(x)}dx}}{\gamma^\gamma(1-\gamma)^{1-\gamma}}\big)^N,
\ \text{as}\ N\to\infty,
\end{aligned}
\end{equation}
with $o(1)$ independent of $\gamma$ and with $O(1)$ uniform over $\gamma$ in compact subintervals of $(0,1)$.
Let $H(\gamma):=\frac{e^{\int_0^\gamma\log\frac{G(x)}{1-G(x)}dx}}{\gamma^\gamma(1-\gamma)^{1-\gamma}}$,
for $\gamma\in(0,1)$. Note that the function $F$, defined in the statement of the theorem, satisfies $F=\log H$.
Since $H$  extends to a continuous, positive function on $[0,1]$, the function $F$ extends to a conintuous function on $[0,1]$.
We leave it to the reader to check that $F'(x)=0$ if and only if $G(x)=x$, and that $F'(x)>0$ for $x$ near 0 and $F'(x)<0$ for $x$ near
1, because $G(0)>0$ and $G(1)<1$.
The claims about $F$ in the statement
of the theorem now follow. If $F$ attains its maximum uniquely at $p^*$,  then in light of
 \eqref{speedN}, \eqref{weights} and the paragraph between \eqref{Stirling} and \eqref{weights},
 the Laplace asymptotic
method gives \eqref{speed*}.
\hfill $\square$

\end{document}